\documentclass{mcom-l}
\newtheorem{theorem}{Theorem}[section]

\theoremstyle{definition}

\theoremstyle{proposition}
\newtheorem{proposition}[theorem]{Proposition}
\theoremstyle{corollary}
\newtheorem{corollary}[theorem]{Corollary}
\theoremstyle{remark}
\newtheorem{remark}[theorem]{Remark}
\numberwithin{equation}{section}
\usepackage[dvips]{graphicx}
\usepackage{url}

\begin{document}

\title{Approximation Errors in Truncated Dimensional Decompositions}

\author{Sharif Rahman}
\address{Applied Mathematical \& Computational Sciences, The University of Iowa,
Iowa City, Iowa 52242}
\email{sharif-rahman@uiowa.edu}
\thanks{The author was supported in part by NSF Grant \#CMMI-0969044 and \#CMMI-1130147.}

%    General info

\date{May 30, 2011 and, in revised form, June 7, 2012.}



\begin{abstract}
The main theme of this paper is error analysis for approximations
derived from two variants of dimensional decomposition of a multivariate
function: the referential dimensional decomposition (RDD) and analysis-of-variance
dimensional decomposition (ADD). New formulae are presented for the
lower and upper bounds of the expected errors committed by bivariately
and arbitrarily truncated RDD approximations when the reference
point is selected randomly, thereby facilitating a means for weighing
RDD against ADD approximations. The formulae reveal that the expected
error from the $S$-variate RDD approximation of a function of $N$
variables, where $0\le S<N<\infty$, is at least $2^{S+1}$ times
greater than the error from the $S$-variate ADD approximation. Consequently,
ADD approximations are exceedingly more precise than RDD approximations. The
analysis also finds the RDD approximation to be sub-optimal for an
arbitrarily selected reference point, whereas the ADD approximation
always results in minimum error. Therefore, the RDD approximation
should be used with caution.
\end{abstract}

\maketitle

\section{Introduction}
Uncertainty quantification of complex systems mandates stochastic
computations of a multivariate output function $y$
that depends on $\mathbf{X}:=(X_{1},\cdots,X_{N})$, a high-dimensional random input with a positive integer $N$. For
practical applications, encountering hundreds of variables or more
is not uncommon, where a function of interest, defined algorithmically
via numerical solution of algebraic, differential, or integral equations,
is all too often expensive to evaluate. Therefore, there is a need to develop
low-dimensional approximations of $y$ by seeking to exploit the hidden
structure potentially lurking underneath a function decomposition.

The dimensional decomposition of
\begin{equation}
y(\mathbf{X})={\displaystyle \sum_{u\subseteq\{1,\cdots,N\}}y_{u}(\mathbf{X}_{u})}\label{1}
\end{equation}
can be viewed as a finite, hierarchical expansion in terms of its
input variables with increasing dimensions, where $u\subseteq\{1,\cdots,N\}$
is a subset with the complementary set $-u=\{1,\cdots,N\}\backslash u$
and cardinality $0\le|u|\le N$, and $y_{u}$ is a $|u|$-variate component function describing a constant or the cooperative influence of $\mathbf{X}_{u}=(X_{i_{1}},\cdots,X_{i_{|u|}})$, $1\leq i_{1}<\cdots<i_{|u|}\leq N$, a subvector of $\mathbf{X}$,
on $y$ when $|u|=0$ or $|u|>0$. The summation in (\ref{1}) comprises $2^{N}$ terms,
with each term depending on a group of variables indexed by a particular
subset of $\{1,\cdots,N\}$, including the empty set $\emptyset$. This decomposition, first presented by
Hoeffding \cite{hoeffding48} in relation to his seminal work on $U$-statistics,
has been studied by many other researchers \cite{owen03}: Sobol \cite{sobol69}
used it for quadrature and analysis of variance (ANOVA) \cite{sobol03}; Efron and Stein \cite{efron81}
applied it to prove their famous lemma on jackknife variances; Owen \cite{owen97}
presented a continuous space version of the nested ANOVA; Hickernell
\cite{hickernell95} developed a reproducing kernel Hilbert space
version; and Rabitz and Alis \cite{rabitz99} made further refinements,
referring to it as high-dimensional model representation (HDMR). More
recently, the author's group formulated this decomposition from the
perspective of Taylor series expansion, solving a number of stochastic-mechanics
problems \cite{rahman04,xu04,xu05}.

In a practical setting, the multivariate function $y$, fortunately,
has an effective dimension \cite{caflisch97} much lower than $N$,
meaning that $y$ can be effectively approximated by a sum of lower-dimensional
component functions $y_{u}$, $|u|\ll N$. Given an integer $0\le S <N$, the truncated dimensional decomposition
\begin{equation}
\hat{y}_{S}(\mathbf{X})={\displaystyle \sum_{{\textstyle {u\subseteq\{1,\cdots,N\}\atop 0\le|u|\le S}}}y_{u}(\mathbf{X}_{u})}\label{2}
\end{equation}
then represents a general $S$-variate approximation
of $y(\mathbf{X})$, which for $S>0$ includes cooperative effects of at
most $S$ input variables $X_{i_{1}},\cdots,X_{i_{S}}$, $1\le i_{1}<\cdots<i_{S}\le N$, on $y$.
However, for (\ref{2}) to be useful, one must ask the
fundamental question: what is the approximation error committed by
$\hat{y}_{S}(\mathbf{X})$ for a given $0\le S<N$? The answer
to this question, however, is neither simple nor unique, because there
are multiple ways to construct the component functions $y_{u}$, $0\le|u|\le S$,
spawning approximations of distinct qualities. Indeed, there exist
two important variants of the decomposition: (1) referential dimensional
decomposition (RDD) and (2) ANOVA dimensional decomposition (ADD),
both representing sums of lower-dimensional component functions of
$y(\mathbf{X})$. While ADD has desirable orthogonal properties,
the ANOVA component functions are difficult to compute due to the high-dimensional
integrals involved. In contrast, the RDD lacks orthogonal features with respect to the probability measure of $\mathbf{X}$,
but its component functions are much easier to obtain. For RDD, an additional
question arises regarding the reference point, which, if improperly
selected, can mar the approximation. Existing error analysis, limited to the univariate truncation of $y(\mathbf{X})$, reveals that the expected error
from the RDD approximation is at least four times larger than the
error from the ADD approximation \cite{wang08}. Although useful to some extent, such result alone is not adequate when evaluating multivariate functions
requiring higher-variate interactions of input \cite{rahman04,xu04,xu05}.
No error estimates exist yet in the current literature even for a
bivariate approximation. Therefore, a more general error analysis
pertaining to a general $S$-variate approximation of a multivariate
function should provide much-needed insights into the mathematical
underpinnings of dimensional decomposition.

The purpose of this paper is twofold. Firstly, a brief exposition
of ADD and RDD is given in Section 2, including clarifications of
parallel developments and synonyms used by various researchers. The
error analysis pertaining to the ADD approximation is described in
Section 3. Secondly, a direct form of the RDD approximation, previously
developed by the author's group, is tapped for providing a vital link
to subsequent error analysis. Section 4 introduces new formulae for
the lower and upper bounds of the expected errors from the bivariate
and general $S$-variate RDD approximations. These error bounds, so
far available only for the univariate approximation, are used
to clarify why ADD approximations are exceedingly more precise than
RDD approximations. There are seven new results stated or proved in this paper: Proposition \ref{p4}, Theorems \ref{t6}, \ref{t7}, and Corollaries \ref{c2}, \ref{c3}, \ref{c4a}, \ref{c4}.
Proofs of other results can be obtained from the references cited, including a longer version of this paper (\url{http://www.engineering.uiowa.edu/~rahman/moc_longpaper.pdf}).
Conclusions are drawn in Section 5.

\section{Dimensional Decomposition}

Let $\mathbb{N}$, $\mathbb{N}_0$, $\mathbb{Z}$, $\mathbb{R}$, and $\mathbb{R}_{0}^{+}$ represent the sets of positive integer (natural), non-negative integer, integer, real, and non-negative real numbers, respectively.  For $k \in \mathbb{N}$, denote by $\mathbb{R}^k$ the $k$-dimensional Euclidean space and by $\mathbb{N}_0^k$ the $k$-dimensional multi-index space. These standard notations will be used throughout the paper.

Let $(\Omega,\mathcal{F},P)$ be a complete probability space, where
$\Omega$ is a sample space, $\mathcal{F}$ is a $\sigma$-field on
$\Omega$, and $P:\mathcal{F}\to[0,1]$ is a probability measure.
With $\mathcal{B}^{N}$ representing the Borel $\sigma$-field on
$\mathbb{R}^{N}$, consider an $\mathbb{R}^{N}$-valued independent
random vector $\mathbf{X}:=(X_{1},\cdots,X_{N}):(\Omega,\mathcal{F})\to(\mathbb{R}^{N},\mathcal{B}^{N})$,
which describes the statistical uncertainties in all system and input
parameters of a given stochastic problem. The probability law of $\mathbf{X}$
is completely defined by its joint probability density function $f_{\mathbf{X}}:\mathbb{R}^{N}\to\mathbb{R}_{0}^{+}$.
Assuming independent coordinates of $\mathbf{X}$, its joint probability
density $f_{\mathbf{X}}(\mathbf{x})=\Pi_{i=1}^{i=N}f_{i}(x_{i})$
can be expressed by a product of marginal probability density functions $f_{i}$ of $X_{i}$, $i=1,\cdots,N$,
defined on the probability triple $(\Omega_{i},\mathcal{F}_{i},P_{i})$
with a bounded or an unbounded support on $\mathbb{R}$.

Consider a non-negative, multiplicative, otherwise general, weight
function $w:\mathbb{R}^{N}\to\mathbb{R}_{0}^{+}$, satisfying $w(\mathbf{x})=\prod_{i=1}^{N}w_{i}(x_{i})$
with the marginal weight functions $w_{i}:\mathbb{R}\to\mathbb{R}_{0}^{+}$, $i=1,\cdots,N$.
For $u\subseteq\{1,\cdots,N\}$, let $w_{-u}(\mathbf{x}_{-u}):=\prod_{i=1,i\notin u}^{N}w_{i}(x_{i})$
define the joint weight function associated with $-u$. Without loss
of generality, assume that the weight functions have been normalized
to integrate to $\int_{\mathbb{R}}w_{i}(x_{i})dx_{i}=\int_{\mathbb{R}^{N-|u|}}w_{-u}(\mathbf{x}_{-u})d\mathbf{x}_{-u}=\int_{\mathbb{R}^{N}}w(\mathbf{x})d\mathbf{x}=1.$
Let $y(\mathbf{X}):=y(X_{1},\cdots,X_{N}$), a real-valued, measurable
transformation on $(\Omega,\mathcal{F})$, define a stochastic response
of interest and $\mathcal{L}_{2}(\Omega,\mathcal{F},P)$ represent a Hilbert
space of square-integrable functions $y$ with respect to the induced
generic measure $w(\mathbf{x})d\mathbf{x}$ supported on $\mathbb{R}^{N}$.
The representation in (\ref{1}) is called dimensional decomposition
if the component functions $y_{u}$, $u\subseteq\{1,\cdots,N\}$, are uniquely
determined from the requirements
\begin{equation}
\int_{\mathbb{R}}y_{u}(\mathbf{x}_{u})w_{i}(x_{i})dx_{i}=0\;\mathrm{for}\; i\in u.\label{3}
\end{equation}
Indeed, integrating (\ref{1}) with respect to the measure
$w_{-u}(\mathbf{x}_{-u})d\mathbf{x}_{-u}$, that is, over all variables except $\mathbf{x}_{u}$,
and using (\ref{3}) yields a recursive form

\begin{equation}
\begin{split}
y(\mathbf{X}) & =  {\displaystyle \sum_{u\subseteq\{1,\cdots,N\}}y_{u}(\mathbf{X}_{u})},\\
y_{\emptyset} & =  \int_{\mathbb{R}^{N}}y(\mathbf{x})w(\mathbf{x})d\mathbf{x},\\
y_{u}(\mathbf{X}_{u}) & =  {\displaystyle \int_{\mathbb{R}^{N-|u|}}y(\mathbf{X}_{u},\mathbf{x}_{-u})}w_{-u}(\mathbf{x}_{-u})d\mathbf{x}_{-u}-{\displaystyle \sum_{v\subset u}}y_{v}(\mathbf{X}_{v}),
\end{split}
\label{3b}
\end{equation}
of the decomposition with $(\mathbf{X}_{u},\mathbf{x}_{-u})$
denoting an $N$-dimensional vector whose $i$th component is $X_{i}$
if $i\in u$ and $x_{i}$ if $i\notin u.$ When $u=\emptyset$, the sum in the last line of (\ref{3b}) vanishes, resulting in the expression of the constant function $y_{\emptyset}$ in the second line. When $u=\{1,\cdots,N\}$, the integration in the last line of (\ref{3b}) is on the empty set, reproducing (\ref{1}) and hence finding the last function $y_{\{1,\cdots,N\}}$. Indeed, all component functions of $y$ can be obtained by interpreting literally the last line of (\ref{3b}). On inversion, (\ref{3b}) results in
\begin{equation}
y(\mathbf{X})={\displaystyle \sum_{u\subseteq\{1,\cdots,N\}}\;\sum_{v\subseteq u}(-1)^{|u|-|v|}{\displaystyle \int_{\mathbb{R}^{N-|v|}}y(\mathbf{X}_{v},\mathbf{x}_{-v})}w_{-v}(\mathbf{x}_{-v})d\mathbf{x}_{-v}},
\label{3c}
\end{equation}
providing an explicit form of the same decomposition.

It is important to emphasize that the measure involved in expressing
the dimensional decomposition in (\ref{3b}) or (\ref{3c}) may
or may not represent the probability measure of $\mathbf{X}$.
Indeed, different measures will create distinct yet exact representations
of $y$, all exhibiting the same structure of (\ref{1}). There
exist two important variants of dimensional decomposition, described
as follows.

\subsection{ADD}

The ADD is generated by selecting the probability measure of $\mathbf{X}$
as the generic measure, that is, $w(\mathbf{x})d\mathbf{x}$=$f_{\mathbf{X}}(\mathbf{x})d\mathbf{x}$
in (\ref{3b}) or (\ref{3c}), yielding the recursive form
\begin{subequations}
\begin{align}
y(\mathbf{X}) & =  {\displaystyle \sum_{u\subseteq\{1,\cdots,N\}}y_{u,A}(\mathbf{X}_{u})}, \label{4a}\\
y_{\emptyset,A} & =  \int_{\mathbb{R}^{N}}y(\mathbf{x})f_{\mathbf{X}}(\mathbf{x})d\mathbf{x},\\
y_{u,A}(\mathbf{X}_{u}) & =  {\displaystyle \int_{\mathbb{R}^{N-|u|}}y(\mathbf{X}_{u},\mathbf{x}_{-u})}f_{-u}(\mathbf{x}_{-u})d\mathbf{x}_{-u}-{\displaystyle \sum_{v\subset u}}y_{v,A}(\mathbf{X}_{v}),
\end{align}
\label{4}
\end{subequations}
\!\!that is commonly found in the ANOVA literature \cite{sobol03,efron81},
although, for the uniform probability measure $d\mathbf{x}$. The explicit version takes the form
\begin{equation}
y(\mathbf{X})={\displaystyle \sum_{u\subseteq\{1,\cdots,N\}}\;\sum_{v\subseteq u}(-1)^{|u|-|v|}{\displaystyle \int_{\mathbb{R}^{N-|v|}}y(\mathbf{X}_{v},\mathbf{x}_{-v})}f_{-v}(\mathbf{x}_{-v})d\mathbf{x}_{-v}},\label{5}
\end{equation}
where $f_{-v}(\mathbf{x}_{-v})=\prod_{i=1,i\notin v}^{N}f_{i}(x_{i})$
is the marginal probability density function of $\mathbf{X}_{-v}:=\mathbf{X}_{\{1,\cdots,N\}\backslash v}$.
Equations (\ref{4}) and (\ref{5}) can also be derived from other perspectives,
including commuting projections on a linear space of real-valued functions
invoked by Kuo \emph{et al.} \cite{kuo10} for the uniform probability
measure.

If $\mathbb{E}$ is the expectation operator with respect to the measure $f_{\mathbf{X}}(\mathbf{x})d\mathbf{x}$,
then two important properties of the ADD component functions, inherited
from (\ref{3}), are as follows.\\

\begin{proposition}
The ADD component functions $y_{u,A}$, $\emptyset\ne u\subseteq\{1,\cdots,N\}$, have zero means, i.e.,
\begin{equation*}
\mathbb{E}\left[y_{u,A}(\mathbf{X}_{u})\right]=\int_{\mathbb{R}^{|u|}}y_{u,A}(\mathbf{x}_{u})f_{u}(\mathbf{x}_{u})d\mathbf{x}_{u}=0.\label{6}
\end{equation*}
\label{p1}
\end{proposition}

\begin{proposition}
Two distinct ADD component functions $y_{u,A}$ and $y_{v,A}$, where $\emptyset\ne u\subseteq\{1,\cdots,N\}$, $\emptyset\ne v\subseteq\{1,\cdots,N\}$,
and $u\neq v$, are orthogonal, i.e., they satisfy the property,
\begin{equation*}
\mathbb{E}\left[y_{u,A}(\mathbf{X}_{u})y_{v,A}(\mathbf{X}_{v})\right]=\int_{\mathbb{R}^{|u\cup v|}}y_{u,A}(\mathbf{x}_{u})y_{v,A}(\mathbf{x}_{v})f_{u\cup v}(\mathbf{x}_{u\cup v})d\mathbf{x}_{u\cup v}=0.\label{7}
\end{equation*}
\label{p2}
\end{proposition}

Traditionally, (\ref{4}) or (\ref{5}) with $X_{j}$, $j=1,\cdots,N$,
following independent, standard uniform distributions, has been identified
as the ANOVA decomposition \cite{sobol03}; however, the author's
recent works \cite{rahman08,rahman09} reveal no fundamental requirement
for a specific probability measure of $\mathbf{X}$, provided
that the resultant integrals in (\ref{4}) or (\ref{5}) exist
and are finite. In this work, the ADD should be interpreted with respect
to an arbitrary but product type probability measure for which it
is always endowed with desirable orthogonal properties. However, the
ADD component functions are difficult to ascertain, because they require
calculation of high-dimensional integrals.

\subsection{RDD}

Consider a reference point $\mathbf{c}=(c_{1},\cdots,c_{N})\in\mathbb{R}^{N}$
and the associated Dirac measure $\prod_{i=1}^{N}\delta(x_{i}-c_{i})dx_{i}$.
The RDD is created when $\prod_{i=1}^{N}\delta(x_{i}-c_{i})dx_{i}$
is chosen as the generic measure in (\ref{3b}), leading to
the recursive form
\begin{subequations}
\begin{align}
y(\mathbf{X}) & =  {\displaystyle \sum_{u\subseteq\{1,\cdots,N\}}y_{u,R}(\mathbf{X}_{u};\mathbf{c})},\label{8a}\\
y_{\emptyset,R} & =  y(\mathbf{c}),\label{8b}\\
y_{u,R}(\mathbf{X}_{u};\mathbf{c}) & =  y(\mathbf{X}_{u},\mathbf{c}_{-u})-{\displaystyle \sum_{v\subset u}}y_{v,R}(\mathbf{X}_{v};\mathbf{c}),\label{8c}
\end{align}
\label{8}
\end{subequations}
\!\!presented as cut-HDMR \cite{rabitz99}, anchored decomposition \cite{wang08,kuo10,hickernell04},
and anchored-ANOVA decomposition \cite{griebel10}, with the latter
two referring to the reference point as the anchor. Xu and Rahman
introduced (\ref{8}) with the aid of Taylor series expansion,
calling it dimension-reduction \cite{xu04} and decomposition \cite{xu05}
methods for calculating statistical moments and reliability of mechanical system responses, respectively.
Again, these various synonyms of the same decomposition exist due
to diverse perspectives employed by researchers in disparate fields.
Analogous to ADD, the RDD can also be described explicitly, for instance \cite{wang08,kuo10,hickernell04},
\begin{equation}
y(\mathbf{X})={\displaystyle \sum_{u\subseteq\{1,\cdots,N\}}\;\sum_{v\subseteq u}(-1)^{|u|-|v|}y(\mathbf{X}_{v},\mathbf{c}_{-v})},\label{9}
\end{equation}
where $(\mathbf{X}_{v},\mathbf{c}_{-v})$ denotes an $N$-dimensional
vector whose $i$th component is $X_{i}$ if $i\in v$ and $c_{i}$
if $i\notin v.$ The second argument ``$\mathbf{c}$'' introduced
in $y_{u,R}$ and $y_{v,R}$ is a reminder that the RDD component
functions depend on the reference point, although $y$ does not, as (\ref{9}) is exact.

An important property of the RDD component functions, also inherited
from (\ref{3}), is as follows \cite{kuo10}.

\begin{proposition}
The RDD component functions $y_{u,R}$,
$\emptyset\ne u\subseteq\{1,\cdots,N\}$, vanish when any of its own
variables $X_{i}$ with $i\in u$ takes on the value of $c_{i}$,
i.e.,
\begin{equation*}
y_{u,R}(\mathbf{X}_{u};\mathbf{c})=0\;\mathrm{whenever\;}X_{i}=c_{i}\;\mathrm{for\; all\;}i\in u.\label{10}
\end{equation*}
\label{p3}
\end{proposition}

Clearly, the RDD component functions lack orthogonal features with
respect to the probability measure of $\mathbf{X}$, but are relatively
easy to obtain as they only involve function evaluations at a chosen
reference point.  However, the RDD component functions can be orthogonal with respect to
other inner products \cite{kuo10}.

\subsection{Truncation}

When a dimensional decomposition, whether ADD or RDD or others, of
a multivariate function is truncated by retaining only lower-dimensional
terms, the result is an approximation. However, due to the special
structure of the decomposition, the approximation is endowed with
an error-minimizing property, described in Theorem \ref{t1}.

\begin{theorem} [\cite{rabitz99}]
For a multivariate function $y:\mathbb{R}^{N}\to\mathbb{R}$ and $0\le S<N$, if
\begin{equation}
\begin{split}
\hat{y}_{S}(\mathbf{X}) & =  {\displaystyle \sum_{{\textstyle {u\subseteq\{1,\cdots,N\}\atop 0\le|u|\le S}}}y_{u}(\mathbf{X}_{u})},\\
y_{\emptyset} & =  \int_{\mathbb{R}^{N}}y(\mathbf{x})w(\mathbf{x})d\mathbf{x},\\
y_{u}(\mathbf{X}_{u}) & =  {\displaystyle \int_{\mathbb{R}^{N-|u|}}y(\mathbf{X}_{u},\mathbf{x}_{-u})}w_{-u}(\mathbf{x}_{-u})d\mathbf{x}_{-u}-{\displaystyle \sum_{v\subset u}}y_{v}(\mathbf{X}_{v}),
\end{split}\label{10b}
\end{equation}
represents an $S$-variate approximation, obtained by truncating
at $0\le|u|\le S$ the dimensional decomposition of $y(\mathbf{X})$
with respect to the generic measure $w(\mathbf{x})d\mathbf{x}$,
then its component functions $y_{u}$, $0\le|u|\le S$, are
uniquely determined from
\begin{equation}
\begin{split}
{\displaystyle \min_{y_{u},0\le|u|\le S}} & \int_{\mathbb{R}^{N}}\biggl[y(\mathbf{x})-{\displaystyle \sum_{{\textstyle {u\subseteq\{1,\cdots,N\}\atop 0\le|u|\le S}}}}y_{u}(\mathbf{x}_{u})\biggr]^{2}w(\mathbf{x})d\mathbf{x},\\
\mathrm{subject\: to} & \int_{\mathbb{R}}y_{u}(\mathbf{x}_{u})w_{i}(x_{i})dx_{i}=0\;\mathrm{for}\; i\in u.\end{split}\label{10c}
\end{equation}
\label{t1}
\end{theorem}

Rabitz and Alis \cite{rabitz99} proved this theorem in pages 202-207 of their paper. The $\mathcal{L}_{2}$ error in (\ref{10c}) committed by a truncated dimensional decomposition is minimized, but only for a specific measure. Given a measure, no other choices of the component functions of $\hat{y}_{S}$
will produce approximations that are better than the one derived from (\ref{10b}). However, different measures will create different
truncated decompositions, resulting in distinct approximation errors.
Therefore, selecting a measure is vitally important for determining
the approximation quality of a dimensional decomposition.

\section{ADD Approximation}

\subsection{Second-Moment Properties}

The $S$-variate ADD approximation $\hat{y}_{S,A}(\mathbf{X})$,
say, of $y(\mathbf{X})$, where $0\le S<N$, is obtained by truncating
the right side of (\ref{4a}) at $0\le|u|\le S$,
yielding
\begin{equation}
\hat{y}_{S,A}(\mathbf{X})={\displaystyle \sum_{{\textstyle {u\subseteq\{1,\cdots,N\}\atop 0\le|u|\le S}}}y_{u,A}(\mathbf{X}_{u})}.\label{11}
\end{equation}
Applying the expectation operator on $y(\mathbf{X})$ and $\hat{y}_{S,A}(\mathbf{X})$
from (\ref{4a}) and (\ref{11}), respectively,
and noting Proposition \ref{p1}, the mean $\mathbb{E}\left[\hat{y}_{S,A}(\mathbf{X})\right]=y_{\emptyset,A}$
of the $S$-variate ADD approximation matches the exact mean $\mathbb{E}\left[y(\mathbf{X})\right]:=\int_{\mathbb{R}^{N}}y(\mathbf{x})f_{\mathbf{X}}(\mathbf{x})d\mathbf{x}=y_{\emptyset,A}$,
regardless of $S$. Applying the expectation operator again, this
time on $\left(\hat{y}_{S,A}(\mathbf{X})-y_{\emptyset,A}\right)^{2}$,
and recognizing Proposition \ref{p2} results in splitting the variance
\begin{equation}
\hat{\sigma}_{S,A}^{2}:=\mathbb{E}\left[\left(\hat{y}_{S,A}(\mathbf{X})-y_{\emptyset,A}\right)^{2}\right]=\sum_{{\textstyle {\emptyset\ne u\subseteq\{1,\cdots,N\}\atop 1\le|u|\le S}}}\sigma_{u}^{2}=\sum_{s=1}^{S}\:\sum_{{\textstyle {\emptyset\ne u\subseteq\{1,\cdots,N\}\atop |u|=s}}}\sigma_{u}^{2}\label{12}
\end{equation}
of the $S$-variate ADD approximation, where $\sigma_{u}^{2}:=\mathbb{E}\left[y_{u,A}^{2}(\mathbf{X}_{u})\right]$
represents the variance of the \emph{zero-}mean ADD component function
$y_{u,A}$, $\emptyset\ne u \subseteq\{1,\cdots,N\}$. Clearly, the approximate variance in (\ref{12})
approaches the exact variance
\begin{equation*}
\sigma^{2}:=\mathbb{E}\left[\left(y(\mathbf{X})-y_{\emptyset,A}\right)^{2}\right]=\sum_{\emptyset\ne u\subseteq\{1,\cdots,N\}}\sigma_{u}^{2}=\sum_{s=1}^{N}\:\sum_{{\textstyle {\emptyset\ne u\subseteq\{1,\cdots,N\}\atop |u|=s}}}\sigma_{u}^{2},\label{13}
\end{equation*}
the sum of all variance terms, when $S\to N$. A normalized version $\sigma_{u}^2/\sigma^2$ is often called the global sensitivity index of $y$ for $\mathbf{X}_u$ \cite{sobol01}.

\subsection{Error from General Approximation}

Define a mean-squared error
\begin{equation}
e_{S,A}:=\mathbb{E}\left[\left(y(\mathbf{X})-\hat{y}_{S,A}(\mathbf{X})\right)^{2}\right]:=\int_{\mathbb{R}^{N}}\left[y(\mathbf{x})-\hat{y}_{S,A}(\mathbf{x})\right]^{2}f_{\mathbf{X}}(\mathbf{x})d\mathbf{x}\label{14}
\end{equation}
committed by the $S$-variate ADD approximation $\hat{y}_{S,A}(\mathbf{X})$ of $y(\mathbf{X})$.
Replacing $y$ and $\hat{y}_{S,A}$ in (\ref{14}) with the
right sides of (\ref{4a}) and (\ref{11}), respectively, and
then recognizing Propositions \ref{p1} and \ref{p2} yields
\begin{equation}
e_{S,A}=\sum_{{\textstyle {\emptyset\ne u\subseteq\{1,\cdots,N\}\atop S+1\le|u|\le N}}}\sigma_{u}^{2}=\sum_{s=S+1}^{N}\:\sum_{{\textstyle {\emptyset\ne u\subseteq\{1,\cdots,N\}\atop |u|=s}}}\sigma_{u}^{2},\label{15}
\end{equation}
which completely eliminates the variance terms of $\sigma^{2}$ that
are associated with $S$- and all lower-variate contributions, an
attractive property of ADD. By setting $S=0,1,2,\cdots$, the error
can be expressed for any truncation of ADD.

\subsection{Optimality}

Among all possible measures, the probability measure endows the ADD approximation with an error-minimizing property, explained as follows.

\begin{proposition}
For a given $S$, the $S$-variate ADD approximation is optimal in the mean-square sense.
\label{p4}
\end{proposition}

\begin{proof}
Consider a generic $S$-variate approximation $\hat{y}_{S}(\mathbf{X})$
of $y(\mathbf{X})$ other than the ADD approximation $\hat{y}_{S,A}(\mathbf{X})$.
Since $y(\mathbf{X})-\hat{y}_{S,A}(\mathbf{X})$ contains
only higher than $S$-variate terms and $\hat{y}_{S,A}(\mathbf{X})-\hat{y}_{S}(\mathbf{X})$
contains at most $S$-variate terms, $y-\hat{y}_{S,A}$
and $\hat{y}_{S,A}-\hat{y}_{S}$ are orthogonal, satisfying $\mathbb{E}\left[\left(y(\mathbf{X})-\hat{y}_{S,A}(\mathbf{X})\right)\left(\hat{y}_{S,A}(\mathbf{X})-\hat{y}_{S}(\mathbf{X})\right)\right]=0$.
Consequently, the second-moment error from an $S$-variate approximation
\begin{align*}
e_{S} & := \mathbb{E}\left[\left(y(\mathbf{X})-\hat{y}_{S}(\mathbf{X})\right)^{2}\right]\\
 & =  \mathbb{E}\left[\left(y(\mathbf{X})-\hat{y}_{S,A}(\mathbf{X})\right)^{2}\right]+\mathbb{E}\left[\left(\hat{y}_{S,A}(\mathbf{X})-\hat{y}_{S}(\mathbf{X})\right)^{2}\right]\\
 & =  e_{S,A}+\mathbb{E}\left[\left(\hat{y}_{S,A}(\mathbf{X})-\hat{y}_{S}(\mathbf{X})\right)^{2}\right]\ge e_{S,A},\label{16}
\end{align*}
proving the mean-square optimality of the $S$-variate ADD approximation.
\end{proof}

Therefore, given a truncation, an RDD approximation, regardless of how the reference point is selected, cannot be better than an ADD approximation for calculating variance.
Further details of RDD approximation errors are described in the next section.

\section{RDD Approximation}

\subsection{Direct Form}

The $S$-variate RDD approximation $\hat{y}_{S,R}(\mathbf{X};\mathbf{c})$,
say, of $y(\mathbf{X})$, where $0\le S<N$, is obtained by truncating
the right side of (\ref{8a}) at $0\le|u|\le S$,
yielding
\begin{equation}
\hat{y}_{S,R}(\mathbf{X};\mathbf{c})={\displaystyle \sum_{{\textstyle {u\subseteq\{1,\cdots,N\}\atop 0\le|u|\le S}}}y_{u,R}(\mathbf{X}_{u};\mathbf{c})},\label{17}
\end{equation}
which depends on the reference point, needing the second argument
``$\mathbf{c}$'' in $\hat{y}_{S,R}$. For error analysis, however,
a suitable direct form of (\ref{17}) is desirable. Theorem
\ref{t2} supplies such a form, which was originally obtained by Xu and Rahman
\cite{xu04} using the Taylor series expansion. The same form was
reported later by Kuo \emph{et al}. \cite{kuo10}.

Let $\mathbf{j}_{k}=(j_{1},\cdots,j_{k})\in\mathbb{N}_0^{k}$, $1\le k\le S$, be a $k$-dimensional multi-index with each component representing a non-negative
integer. The multi-index, used in Theorem \ref{t2}, obeys the following
standard notations: (1) $|\mathbf{j}_{k}|=\sum_{p=1}^{p=k}j_{p}$;
(2) $\mathbf{j}_{k}!=\prod_{p=1}^{p=k}j_{p}!$; (3) ${\displaystyle \partial^{\mathbf{j}_{k}}y(\mathbf{c})=\partial^{j_{1}+\cdots+j_{k}}y(\mathbf{c})/\partial X_{i_{1}}^{j_{1}}\cdots\partial X_{i_{k}}^{j_{k}}}$;
(4) $(\mathbf{X}_{u}-\mathbf{c}_{u})^{\mathbf{j}_{k}}=\prod_{p=1}^{p=k}(X_{i_{p}}-c_{i_{p}})^{j_{p}}$,
$1\le i_{1}<\cdots<i_{k}\le N$.

\begin{theorem} [Multivariate Function Theorem \cite{xu04}]
For a differentiable multivariate
function $y:\mathbb{R}^{N}\to\mathbb{R}$ and $0\le S<N$, if
\begin{equation}
\hat{y}_{S,R}(\mathbf{X};\mathbf{c})={\displaystyle \sum_{k=0}^{S}}(-1)^{k}\binom{N-S+k-1}{k}\sum_{{\textstyle {u\subseteq\{1,\cdots,N\}\atop |u|=S-k}}}y(\mathbf{X}_{u},\mathbf{c}_{-u})\label{18}
\end{equation}
represents an $S$-variate RDD approximation of $y(\mathbf{X})$,
then $\hat{y}_{S,R}(\mathbf{X};\mathbf{c})$ consists of all
terms of the Taylor series expansion of $y(\mathbf{X})$ at $\mathbf{c}$
that have less than or equal to $S$ variables, i.e.,
\begin{equation*}
\hat{y}_{S,R}(\mathbf{X};\mathbf{c})={\displaystyle \sum_{k=0}^{S}}t_{k},\label{19}
\end{equation*}
where
\begin{align*}
t_{0} & =  y(\mathbf{c}),\\
t_{k} & =  {\displaystyle \sum_{{\textstyle {\mathbf{j}_{k} \in \mathbb{N}_0^k \atop j_1,\cdots,j_k \neq 0}}}}{\displaystyle \frac{1}{\mathbf{j}_{k}!}{\displaystyle \sum_{{\textstyle {\emptyset\ne u\subseteq\{1,\cdots,N\}\atop |u|=k}}}}}\partial^{\mathbf{j}_{k}}y(\mathbf{c})\left(\mathbf{X}_{u}-\mathbf{c}_{u}\right)^{\mathbf{j}_{k}};\;1\le k\le S.\label{20}
\end{align*}
\label{t2}
\end{theorem}

Xu and Rahman \cite{xu04} proved this theorem in pages 1996-2000 of their paper when $\mathbf{c}=\mathbf{0}$ without loss of generality. The stochastic method associated with
the RDD approximation was simply called {}``decomposition method'' \cite{xu05}. Theorem \ref{t2} implies that the RDD approximation
$\hat{y}_{S,R}(\mathbf{X};\mathbf{c})$ in (\ref{18}), when compared with the Taylor series expansion of $y(\mathbf{X})$,
yields residual error that includes only terms of dimensions $S+1$
and higher. All higher-order $S$- and lower-variate terms of $y(\mathbf{X})$
are included in (\ref{18}), which should therefore generally
provide a higher-order approximation of a multivariate function than the
equation derived from an $S$-order Taylor expansion.

Equations (\ref{17}) and (\ref{18}) both follow the same structure of (\ref{2}). However, due to the distinct perspectives involved,
it is not obvious if these equations represent the same function $\hat{y}_{S,R}$. A lemma and a theorem recently proved by the author demonstrate that, indeed, they do \cite{rahman11}.

\begin{corollary}
When $S=0$, $1$, and $2$, (\ref{18}) degenerates to the zero-variate RDD approximation
\begin{equation}
\hat{y}_{0,R}(\mathbf{X};\mathbf{c})=y(\mathbf{c}),\label{25b}
\end{equation}
the univariate RDD approximation
\begin{equation}
\hat{y}_{1,R}(\mathbf{X};\mathbf{c})={\displaystyle \sum_{i=1}^{N}y(X_{i},\mathbf{c}_{-\{i\}})}-(N-1)y(\mathbf{c}),\label{26}
\end{equation}
and the bivariate RDD approximation
\begin{equation}
\begin{split}
\hat{y}_{2,R}(\mathbf{X};\mathbf{c}) & =  {\displaystyle \sum_{i=1}^{N-1}\sum_{j=i+1}^{N}}y(X_{i},X_{j},\mathbf{c}_{-\{i,j\}}){\displaystyle -(N-2){\displaystyle \sum_{i=1}^{N}y(X_{i},\mathbf{c}_{-\{i\}})}}\\
 &   \;\;\;\; +{\displaystyle \frac{1}{2}}{\displaystyle (N-1)(N-2)}y(\mathbf{c}),
 \end{split}\label{27}
\end{equation}
respectively.
\label{c0}
\end{corollary}

\begin{remark}
Since the right side of (\ref{26}) comprises
only univariate functions, the interpolation or integration of $\hat{y}_{1,R}(\mathbf{X};\mathbf{c})$
is essentially univariate. Similarly, the right side of (\ref{27}),
which contains at most bivariate functions, requires at most bivariate
interpolation or integration of $\hat{y}_{2,R}(\mathbf{X};\mathbf{c})$.
Therefore, appellation of the terms {}``univariate approximation'' and {}``bivariate approximation''
for $\hat{y}_{1,R}(\mathbf{X};\mathbf{c})$ in (\ref{26})
and $\hat{y}_{2,R}(\mathbf{X};\mathbf{c})$ in (\ref{27}),
respectively, is more appropriate than referring to them as first-order
and second-order approximations.
\end{remark}

\subsection{Expected Error}

Following similar consideration, define another mean-squared error
\begin{equation}
\begin{split}
e_{S,R}(\mathbf{c})&
:=\mathbb{E}\left[\left(y(\mathbf{X})-\hat{y}_{S,R}(\mathbf{X};\mathbf{c})\right)^{2}\right]
:=\int_{\mathbb{R}^{N}}\left[y(\mathbf{x})-\hat{y}_{S,R}(\mathbf{x};\mathbf{c})\right]^{2}f_{\mathbf{X}}(\mathbf{x})d\mathbf{x}
\end{split}\label{28}
\end{equation}
associated with the $S$-variate RDD approximation $\hat{y}_{S,R}(\mathbf{X};\mathbf{c})$
of $y(\mathbf{X})$, which depends on the reference point $\mathbf{c}$.
Wang \cite{wang08} suggested choosing a random reference point uniformly
distributed over $[0,1]^{N}$ and then calculating the error on average.
But, $\mathbf{X}$ defined here may follow an arbitrary probability
law with density $f_{\mathbf{X}}(\mathbf{x})$; therefore,
selecting the reference point characterized by the probability density
$f_{\mathbf{X}}(\mathbf{c})$ is more appropriate, which leads
to
\begin{equation}
\begin{split}
\mathbb{E}\left[e_{S,R}(\mathbf{c})\right] & :=\int_{\mathbb{R}^{N}}e_{S,R}(\mathbf{c})f_{\mathbf{X}}(\mathbf{c})d\mathbf{c} \\
&
=\int_{\mathbb{R}^{2N}}\left[y(\mathbf{x})-\hat{y}_{S,R}(\mathbf{x};\mathbf{c})\right]^{2}f_{\mathbf{X}}(\mathbf{x})f_{\mathbf{X}}(\mathbf{c})d\mathbf{x}d\mathbf{c},
\end{split}\label{29}
\end{equation}
as the expected value of the RDD error. Simplifying (\ref{29})
in terms of the variance components of $y$, as done for the ADD error in (\ref{15}),
for arbitrary $S$ and $N$ may appear formidable. Here, the \emph{zero}-variate
($S=0$), univariate ($S=1$), and bivariate ($S=2$) approximation
errors for arbitrary $N$ will be derived first, followed by error
analysis for a general $S$-variate approximation. In all cases, the
derivations require using (1) the relationships,
\begin{subequations}
\begin{align}
y_{\emptyset,A} & = \int_{\mathbb{R}^{N}}y(\mathbf{c})f_{\mathbf{X}}(\mathbf{c})d\mathbf{c},\label{30a}\\
y_{u,A}(\mathbf{x}_{u}) & = \int_{\mathbb{R}^{N}}y_{u,R}(\mathbf{x}_{u};\mathbf{c})f_{\mathbf{X}}(\mathbf{c})d\mathbf{c},\label{30b}\\
\hat{y}_{S,A}(\mathbf{x}) & = \int_{\mathbb{R}^{N}}\hat{y}_{S,R}(\mathbf{x};\mathbf{c})f_{\mathbf{X}}(\mathbf{c})d\mathbf{c} \label{30c},
\end{align}
\label{30}
\end{subequations}
\!\!that exist between ADD and RDD component functions and approximations
and (2) Sobol's formula \cite{sobol01},
\begin{equation}
D_{u}:=\sum_{v\subseteq u}\sigma_{v}^{2}=\int_{\mathbb{R}^{2N-|u|}}y(\mathbf{x})y(\mathbf{x}_{u},\mathbf{c}_{-u})f_{\mathbf{X}}(\mathbf{x})f_{\mathbf{X}_{-u}}(\mathbf{c}_{-u})d\mathbf{x}d\mathbf{c}_{-u}-y_{\emptyset,A}^{2},\label{31}
\end{equation}
for select choices of $u$ described in the following subsection.
Equations (\ref{30a}) and (\ref{30b}) follow from Propositions \ref{p1}, \ref{p2}, and \ref{p3} and definitions of respective component functions in (\ref{4}) and (\ref{8}), eventually
leading to (\ref{30c}). The term $D_{u}$ in Sobol's formula represents a sum of variance
terms contributed by the ADD component functions that belong to $\emptyset\ne u\subseteq\{1,\cdots,N\}$.

\subsection{Expected Errors from Zero-variate, Univariate, and Bivariate Approximations}

Theorems \ref{t4}, \ref{t5}, and \ref{t6} show how the expected errors from the \emph{zero}-variate,
univariate, and bivariate RDD approximations, respectively, depend
on the variance components of $y$.

\begin{theorem}
Let $\mathbf{c}=(c_{1},\cdots,c_{N})\in\mathbb{R}^{N}$
be a random vector with the joint probability density function of
the form $f_{\mathbf{X}}(\mathbf{c})=\prod_{j=1}^{j=N}f_{j}(c_{j})$,
where $f_{j}$ is the marginal probability density function of its
$j$th coordinate. Then the expected error committed by the zero-variate
RDD approximation for $1\le N<\infty$ is
\begin{equation*}
\mathbb{E}\left[e_{0,R}(\mathbf{c})\right]=2\sigma^{2},\label{31b}
\end{equation*}
where $\sigma^{2}:=\mathbb{E}\left[\left(y(\mathbf{X})-y_{\emptyset,A}\right)^{2}\right]=\mathbb{E}\left[y^2(\mathbf{X})\right]-y_{\emptyset,A}^2$
is the variance of $y$.
\label{t4}
\end{theorem}

\begin{proof}
Setting $S=0$ in (\ref{29}) and using
the expression of $\hat{y}_{0,R}(\mathbf{x};\mathbf{c})$
from (\ref{25b}), the expected error from the \emph{zero}-variate RDD approximation becomes
\begin{equation*}
\begin{split}
\mathbb{E}\left[e_{0,R}(\mathbf{c})\right] & = \int_{\mathbb{R}^{2N}}\left[y(\mathbf{x})-y(\mathbf{c})\right]^{2}f_{\mathbf{X}}(\mathbf{x})f_{\mathbf{X}}(\mathbf{c})d\mathbf{x}d\mathbf{c}\\
 & = \int_{\mathbb{R}^{N}}y^{2}(\mathbf{x})f_{\mathbf{X}}(\mathbf{x})d\mathbf{x}+\int_{\mathbb{R}^{N}}y^{2}(\mathbf{c})f_{\mathbf{X}}(\mathbf{c})d\mathbf{c}\\
 & \;\;\;\;   -2\int_{\mathbb{R}^{N}}y(\mathbf{x})f_{\mathbf{X}}(\mathbf{x})d\mathbf{x}\int_{\mathbb{R}^{N}}y(\mathbf{c})f_{\mathbf{X}}(\mathbf{c})d\mathbf{c}\\
 & =  \sigma^{2}+y_{\emptyset,A}^{2}+\sigma^{2}+y_{\emptyset,A}^{2}-2y_{\emptyset,A}^{2}\\
 & =  2\sigma^{2},
 \end{split}\label{31c}
\end{equation*}
where the third equality exploits the ADD-RDD relationship in (\ref{30a}). Hence, the theorem is proven.
\end{proof}

\begin{theorem}
Let $\mathbf{c}=(c_{1},\cdots,c_{N})\in\mathbb{R}^{N}$
be a random vector with the joint probability density function of
the form $f_{\mathbf{X}}(\mathbf{c})=\prod_{j=1}^{j=N}f_{j}(c_{j})$,
where $f_{j}$ is the marginal probability density function of its
$j$th coordinate. Then the expected error committed by the univariate
RDD approximation for $2\le N<\infty$ is
\begin{equation}
\mathbb{E}\left[e_{1,R}(\mathbf{c})\right]={\displaystyle \sum_{s=2}^{N}\left(s^{2}-s+2\right)\sum_{{\textstyle {\emptyset\neq u\subseteq\{1,\cdots,N\}\atop |u|=s}}}}\sigma_{u}^{2},\label{32}
\end{equation}
where $\sigma_{u}^{2}=\mathbb{E}\left[y_{u,A}^{2}(\mathbf{X}_{u})\right]$ is the variance of the \emph{zero}-mean ADD component function $y_{u,A}$, $\emptyset\ne u\subseteq\{1,\cdots,N\}$.
\label{t5}
\end{theorem}

\begin{proof}
Setting $S=1$ in (\ref{29}), the expected
error from the univariate RDD approximation on expansion is a sum
\begin{equation*}
\mathbb{E}\left[e_{1,R}(\mathbf{c})\right]=I_{1,1}+I_{1,2}+I_{1,3}\label{33}
\end{equation*}
of three integrals
\begin{equation*}
\begin{split}
I_{1,1} & := \int_{\mathbb{R}^{2N}}y^{2}(\mathbf{x})f_{\mathbf{X}}(\mathbf{x})f_{\mathbf{X}}(\mathbf{c})d\mathbf{x}d\mathbf{c},\\
I_{1,2} & :=  -2\int_{\mathbb{R}^{2N}}y(\mathbf{x})\hat{y}_{1,R}(\mathbf{x};\mathbf{c})f_{\mathbf{X}}(\mathbf{x})f_{\mathbf{X}}(\mathbf{c})d\mathbf{x}d\mathbf{c},\\
I_{1,3} & :=  \int_{\mathbb{R}^{2N}}\hat{y}_{1,R}^{2}(\mathbf{x};\mathbf{c})f_{\mathbf{X}}(\mathbf{x})f_{\mathbf{X}}(\mathbf{c})d\mathbf{x}d\mathbf{c}
\end{split}\label{34}
\end{equation*}
on $\mathbb{R}^{2N}$, where their first indices represent the univariate
approximation. The first integral
\begin{equation}
I_{1,1}=\mathbb{E}\left[y^{2}(\mathbf{X})\right]=y_{\emptyset,A}^{2}+\sigma^{2}=y_{\emptyset,A}^{2}+\sum_{s=1}^{N}\:\sum_{{\textstyle {\emptyset\ne u\subseteq\{1,\cdots,N\}\atop |u|=s}}}\sigma_{u}^{2},\label{35}
\end{equation}
expressed in terms of the variance components, is independent of $S$.
However, the second integral depends on $S$, yielding
\begin{equation}
\begin{split}
I_{1,2} & = -2\int_{\mathbb{R}^{N}}y(\mathbf{x})\hat{y}_{1,A}(\mathbf{x})f_{\mathbf{X}}(\mathbf{x})d\mathbf{x}\\
 & = -2\int_{\mathbb{R}^{N}}\hat{y}_{1,A}^{2}(\mathbf{x})f_{\mathbf{X}}(\mathbf{x})d\mathbf{x}\\
 & =  -2\mathbb{E}\left[\hat{y}_{1,A}^{2}(\mathbf{X})\right] \\
 & =  -2\left(y_{\emptyset,A}^{2}+\hat{\sigma}_{1,A}^{2}\right) \\
 & =  -2y_{\emptyset,A}^{2}-2{\displaystyle \sum_{{\textstyle {\emptyset\ne u\subseteq\{1,\cdots,N\}\atop |u|=1}}}}\sigma_{u}^{2},
\end{split}\label{36}
\end{equation}
where the first, second, and fifth lines are obtained (1) employing the
ADD-RDD relationships in (\ref{30c}) for $S=1$, (2) recognizing
$y-\hat{y}_{1,A}$ and $\hat{y}_{1,A}$ to be orthogonal, satisfying $\mathbb{E}\left[\left(y(\mathbf{X})-\hat{y}_{1,A}(\mathbf{X})\right)\hat{y}_{1,A}(\mathbf{X})\right]=0$, and (3) applying (\ref{12}) for $S=1$, respectively.
Using the expression of $\hat{y}_{1,R}(\mathbf{x};\mathbf{c})$
from (\ref{26}) and noting independent coordinates of $\mathbf{X}$
and $\mathbf{c}$, the expanded third integral becomes
\begin{equation}
\begin{split}
I_{1,3} & =  \int_{\mathbb{R}^{2N}}{\displaystyle \biggl[\sum_{i=1}^{N}y(x_{i},\mathbf{c}_{-\{i\}})}-(N-1)y(\mathbf{c})\biggr]^{2}f_{\mathbf{X}}(\mathbf{x})f_{\mathbf{X}}(\mathbf{c})d\mathbf{x}d\mathbf{c}\\
 & =  \int_{\mathbb{R}^{2N}}\biggl[{\displaystyle \sum_{i=1}^{N}}y^{2}(x_{i},\mathbf{c}_{-\{i\}})+2{\displaystyle \sum_{i=1}^{N-1}\sum_{j=i+1}^{N}}y(x_{i},\mathbf{c}_{-\{i\}})y(x_{j},\mathbf{c}_{-\{j\}})\\
 &  \;\;\; +(N-1)^{2}y^{2}(\mathbf{c})-2(N-1){\displaystyle \sum_{i=1}^{N}}y(X_{i},\mathbf{c}_{-\{i\}})y(\mathbf{c})\biggr]f_{\mathbf{X}}(\mathbf{x})f_{\mathbf{X}}(\mathbf{c})d\mathbf{x}d\mathbf{c}.
\end{split}\label{37}
\end{equation}
Further evaluation of this integral requires exploiting Sobol's formula
in (\ref{31}) for $u=-\{i\}$ and $u=-\{i,j\}$, where $i,j=1,\cdots,N$,
$i<j$, yielding \vspace{-0.02in}
\begin{equation}
\begin{split}
I_{1,3} & = (N^{2}-N+1)\left(\sigma^{2}+y_{\emptyset,A}^{2}\right)-2(N-1){\displaystyle \sum_{i=1}^{N}\left(D_{-\{i\}}+y_{\emptyset,A}^{2}\right)}\\
 &   \;\;\; +2{\displaystyle \sum_{i=1}^{N-1}\sum_{j=i+1}^{N}}\left(D_{-\{i,j\}}+y_{\emptyset,A}^{2}\right)\\
 & =  y_{\emptyset,A}^{2}+{\displaystyle \sum_{{\textstyle {\emptyset\ne u\subseteq\{1,\cdots,N\}\atop |u|=1}}}}\sigma_{u}^{2}+3{\displaystyle \sum_{{\textstyle {\emptyset\ne u\subseteq\{1,\cdots,N\}\atop |u|=2}}}}\sigma_{u}^{2}+\cdots\\
 &   \;\;\; +\left(s^{2}-s+1\right){\displaystyle \sum_{{\textstyle {\emptyset\ne u\subseteq\{1,\cdots,N\}\atop |u|=s}}}}\sigma_{u}^{2}+\cdots+\left(N^{2}-N+1\right){\displaystyle \sum_{{\textstyle {\emptyset\ne u\subseteq\{1,\cdots,N\}\atop |u|=N}}}}\sigma_{u}^{2}\\
 & =  b_{1}(0)y_{\emptyset,A}^{2}+{\displaystyle \sum_{s=1}^{N}b_{1}(s)\sum_{{\textstyle {\emptyset\ne u\subseteq\{1,\cdots,N\}\atop |u|=s}}}}\sigma_{u}^{2},
\end{split}\label{37b}
\end{equation}

\vspace{-0.09in}
\noindent where
\begin{equation*}
b_{1}(s)=s^{2}-s+1,\;0\le s\le N,\label{38}
\end{equation*}
is the generic $s$-variate coefficient for the univariate approximation,
obtained by counting the number of $y_{\emptyset,A}^2$ or $\sigma_{u}^{2}$
for $|u|=s$ $-$ \emph{e.g.}, $\sigma_{1}^{2}$ for $s=1$,
$\sigma_{12}^{2}$ for $s=2$, and so on $-$ due to symmetry. Adding
all terms in (\ref{35}), (\ref{36}), and (\ref{37b}),
with the recognition that $b_{1}(0)=b_{1}(1)=1$, yields (\ref{32}),
completing the proof.
\end{proof}

\begin{theorem}
Let $\mathbf{c}=(c_{1},\cdots,c_{N})\in\mathbb{R}^{N}$
be a random vector with the joint probability density function of
the form $f_{\mathbf{X}}(\mathbf{c})=\prod_{j=1}^{j=N}f_{j}(c_{j})$,
where $f_{j}$ is the marginal probability density function of its
$j$th coordinate. Then the expected error committed by the bivariate
RDD approximation for $3\le N<\infty$ is
\begin{equation}
\mathbb{E}\left[e_{2,R}(\mathbf{c})\right]={\displaystyle \sum_{s=3}^{N}{\displaystyle \frac{1}{4}}\left(s^{4}-2s^{3}-s^{2}+2s+8\right)\sum_{{\textstyle {\emptyset\neq u\subseteq\{1,\cdots,N\}\atop |u|=s}}}}\sigma_{u}^{2},\label{39}
\end{equation}
where $\sigma_{u}^{2}=\mathbb{E}\left[y_{u,A}^{2}(\mathbf{X}_{u})\right]$ is the variance of
the \emph{zero}-mean ADD component function $y_{u,A}$,
$\emptyset\ne u\subseteq\{1,\cdots,N\}$.
\label{t6}
\end{theorem}

\begin{proof}
Setting $S=2$ in (\ref{29}), the expected
error from the bivariate RDD approximation on expansion is another
sum
\begin{equation*}
\mathbb{E}\left[e_{2,R}(\mathbf{c})\right]=I_{2,1}+I_{2,2}+I_{2,3}\label{40}
\end{equation*}
of three $2N$-dimensional integrals
\begin{equation*}
\begin{split}
I_{2,1} & := \int_{\mathbb{R}^{2N}}y^{2}(\mathbf{x})f_{\mathbf{X}}(\mathbf{x})f_{\mathbf{X}}(\mathbf{c})d\mathbf{x}d\mathbf{c},\\
I_{2,2} & := -2\int_{\mathbb{R}^{2N}}y(\mathbf{x})\hat{y}_{2,R}(\mathbf{x};\mathbf{c})f_{\mathbf{X}}(\mathbf{x})f_{\mathbf{X}}(\mathbf{c})d\mathbf{x}d\mathbf{c},\\
I_{2,3} & :=  \int_{\mathbb{R}^{2N}}\hat{y}_{2,R}^{2}(\mathbf{x};\mathbf{c})f_{\mathbf{X}}(\mathbf{x})f_{\mathbf{X}}(\mathbf{c})d\mathbf{x}d\mathbf{c},
\end{split}\label{41}
\end{equation*}
where their first indices represent the bivariate approximation. Since
the first integral does not depend on $S$,
\begin{equation}
I_{2,1}=y_{\emptyset,A}^{2}+\sigma^{2}=y_{\emptyset,A}^{2}+\sum_{s=1}^{N}\:\sum_{{\textstyle {\emptyset\ne u\subseteq\{1,\cdots,N\}\atop |u|=s}}}\sigma_{u}^{2}\label{42}
\end{equation}
is the same as $I_{1,1}$. Following a similar reasoning employed
for the univariate approximation, the second integral
\begin{equation}
\begin{split}
I_{2,2} & = -2\int_{\mathbb{R}^{N}}y(\mathbf{x})\hat{y}_{2,A}(\mathbf{x})f_{\mathbf{X}}(\mathbf{x})d\mathbf{x}\\
 & = -2\int_{\mathbb{R}^{N}}\hat{y}_{2,A}^{2}(\mathbf{x})f_{\mathbf{X}}(\mathbf{x})d\mathbf{x}\\
 & =  -2\mathbb{E}\left[\hat{y}_{2,A}^{2}(\mathbf{X})\right] = -2\left(y_{\emptyset,A}^{2}+\hat{\sigma}_{2,A}^{2}\right) \\
% & =  -2\left(y_{\emptyset,A}^{2}+\hat{\sigma}_{2,A}^{2}\right)\\
 & =  -2y_{\emptyset,A}^{2}-2{\displaystyle \sum_{{\textstyle {\emptyset\ne u\subseteq\{1,\cdots,N\}\atop |u|=1}}}}\sigma_{u}^{2}-2{\displaystyle \sum_{{\textstyle {\emptyset\ne u\subseteq\{1,\cdots,N\}\atop |u|=2}}}}\sigma_{u}^{2}
\end{split}\label{43}
\end{equation}
contains variance terms associated with at most two variables. Using
the expression of $\hat{y}_{2,R}(\mathbf{x};\mathbf{c})$
from (\ref{27}), the expanded third integral becomes
\begin{equation*}
\begin{split}
I_{2,3} & =  \int_{\mathbb{R}^{2N}}\biggl[{\displaystyle \sum_{i=1}^{N-1}\sum_{j=i+1}^{N}}y(x_{i},x_{j},\mathbf{c}_{-\{i,j\}}){\displaystyle -(N-2){\displaystyle \sum_{i=1}^{N}y(x_{i},\mathbf{c}_{-\{i\}})}}\\
 &  \;\;\; +{\displaystyle \frac{1}{2}(N-1)(N-2)}y(\mathbf{c})\biggr]^{2}f_{\mathbf{X}}(\mathbf{x})f_{\mathbf{X}}(\mathbf{c})d\mathbf{x}d\mathbf{c}\\
 & =  \int_{\mathbb{R}^{2N}}\biggl\{\biggl[{\displaystyle \sum_{i=1}^{N-1}\sum_{j=i+1}^{N}}y(x_{i},x_{j},\mathbf{c}_{-\{i,j\}})\biggr]^{2}{\displaystyle +(N-2)^{2}{\displaystyle \biggl[\sum_{i=1}^{N}y(x_{i},\mathbf{c}_{-\{i\}})\biggr]^{2}}}\\
 &  \;\;\; +{\displaystyle \frac{1}{4}(N-1)^{2}(N-2)^{2}}y^{2}(\mathbf{c})\\
 &  \;\;\; -2(N-2)\biggl[{\displaystyle \sum_{i=1}^{N-1}\sum_{j=i+1}^{N}}y(x_{i},x_{j},\mathbf{c}_{-\{i,j\}})\biggr]\biggl[{\displaystyle \sum_{i=1}^{N}}y(x_{i},\mathbf{c}_{-\{i\}})\biggr]\\
 &  \;\;\; -(N-1)(N-2)^{2}\biggl[{\displaystyle \sum_{i=1}^{N}}y(x_{i},\mathbf{c}_{-\{i\}})\biggr]y(\mathbf{c})\\
 &  \;\;\; +(N-1)(N-2)\biggl[{\displaystyle \sum_{i=1}^{N-1}\sum_{j=i+1}^{N}}y(x_{i},x_{j},\mathbf{c}_{-\{i,j\}})\biggr]y(\mathbf{c})\biggr\} f_{\mathbf{X}}(\mathbf{x})f_{\mathbf{X}}(\mathbf{c})d\mathbf{x}d\mathbf{c}.
\end{split}\label{43b}
\end{equation*}
Employing Sobol's formula, this time for $u=-\{i\}$, $u=-\{j\}$,
$u=-\{i,j\}$, $u=-\{i,k\}$, $u=-\{j,k\}$, $u=-\{i,j,k\}$, and
$u=-\{i,j,k,l\}$, where $i,j,k,l=1,\cdots,N$, $i<j<k<l$, results in
\begin{equation}
\begin{split}
I_{2,3} & =  {\displaystyle \frac{N(N-1)}{2}}\left(\sigma^{2}+y_{\emptyset,A}^{2}\right)\\
 &  \;\;\; +2{\displaystyle \sum_{i=1}^{N-2}\sum_{j=i+1}^{N-1}\sum_{k=j+1}^{N}}\left(D_{-\{i,j\}}+D_{-\{i,k\}}+D_{-\{j,k\}}+3y_{\emptyset,A}^{2}\right)\\
 &  \;\;\; +6{\displaystyle \sum_{i=1}^{N-3}\sum_{j=i+1}^{N-2}\sum_{k=j+1}^{N-1}}{\displaystyle \sum_{l=k+1}^{N}}\left(D_{-\{i,j,k,l\}}+y_{\emptyset,A}^{2}\right)+N(N-2)^{2}{\displaystyle \left(\sigma^{2}+y_{\emptyset,A}^{2}\right)}\\
 &  \;\;\; +2(N-2)^{2}{\displaystyle \sum_{i=1}^{N-1}\sum_{j=i+1}^{N}}\left(D_{-\{i,j\}}+y_{\emptyset,A}^{2}\right)+{\displaystyle {\displaystyle \frac{1}{4}}(N-1)^{2}(N-2)^{2}}{\displaystyle \left(\sigma^{2}+y_{\emptyset,A}^{2}\right)}\\
 &  \;\;\; -2(N-2){\displaystyle \sum_{i=1}^{N-1}\sum_{j=i+1}^{N}}\left(D_{-\{i\}}+D_{-\{j\}}+2y_{\emptyset,A}^{2}\right)\\
 &  \;\;\; -6(N-2){\displaystyle \sum_{i=1}^{N-2}\sum_{j=i+1}^{N-1}}{\displaystyle \sum_{k=j+1}^{N}}\left(D_{-\{i,j,k\}}+y_{\emptyset,A}^{2}\right)\\
 &  \;\;\; -(N-1)(N-2)^{2}{\displaystyle \sum_{i=1}^{N}\left(D_{-\{i\}}+y_{\emptyset,A}^{2}\right)}\\
 &  \;\;\; +(N-1)(N-2){\displaystyle \sum_{i=1}^{N-1}\sum_{j=i+1}^{N}}\left(D_{-\{i,j\}}+y_{\emptyset,A}^{2}\right)\\
 & =  y_{\emptyset,A}^{2}+{\displaystyle \sum_{{\textstyle {\emptyset\ne u\subseteq\{1,\cdots,N\}\atop |u|=1}}}}\sigma_{u}^{2}+{\displaystyle \sum_{{\textstyle {\emptyset\ne u\subseteq\{1,\cdots,N\}\atop |u|=2}}}}\sigma_{u}^{2}+7{\displaystyle \sum_{{\textstyle {\emptyset\ne u\subseteq\{1,\cdots,N\}\atop |u|=3}}}}\sigma_{u}^{2}\\
 &  \;\;\; +\cdots+{\displaystyle \frac{1}{4}}\left(s^{4}-2s^{3}-s^{2}+2s+4\right){\displaystyle \sum_{{\textstyle {\emptyset\ne u\subseteq\{1,\cdots,N\}\atop |u|=s}}}}\sigma_{u}^{2}\\
 &  \;\;\; +\cdots+{\displaystyle \frac{1}{4}}\left(N^{4}-2N^{3}-N^{2}+2N+4\right){\displaystyle \sum_{{\textstyle {\emptyset\ne u\subseteq\{1,\cdots,N\}\atop |u|=N}}}}\sigma_{u}^{2}\\
 & =  b_{2}(0)y_{\emptyset,A}^{2}+{\displaystyle \sum_{s=1}^{N}b_{2}(s)\sum_{{\textstyle {\emptyset\ne u\subseteq\{1,\cdots,N\}\atop |u|=s}}}}\sigma_{u}^{2},
\end{split}\label{44}
\end{equation}
producing the generic $s$-variate coefficient
\begin{equation*}
b_{2}(s)=\frac{1}{4}\left(s^{4}-2s^{3}-s^{2}+2s+4\right),\;0\le s\le N,\label{45}
\end{equation*}
for the bivariate approximation. Adding all terms in (\ref{42}), (\ref{43}), and (\ref{44}), with the understanding that $b_{2}(0)=b_{2}(1)=b_{2}(2)=1$,
yields (\ref{39}), proving the theorem.
\end{proof}

\begin{corollary}
The expected error $\mathbb{E}\left[e_{0,R}\right]$
from the zero-variate RDD approximation, expressed in terms of the
error $e_{0,A}$ from the zero-variate ADD approximation, is
\begin{equation*}
\mathbb{E}\left[e_{0,R}\right]=2e_{0,A},\;1\le N<\infty.\label{45b}
\end{equation*}
\label{c1}
\end{corollary}

\begin{corollary}
The lower and upper bounds of the
expected errors $\mathbb{E}\left[e_{1,R}\right]$ and $\mathbb{E}\left[e_{2,R}\right]$
from the univariate and bivariate RDD approximations, respectively,
expressed in terms of the errors $e_{1,A}$ and $e_{2,A}$
from the univariate and bivariate ADD approximations, are
\begin{equation*}
4e_{1,A} \le  \mathbb{E}\left[e_{1,R}\right]  \le  \left(N^{2}-N+2\right)e_{1,A},\;2\le N<\infty,
\end{equation*}
and
\begin{equation*}
8e_{2,A} \le \mathbb{E}\left[e_{2,R}\right]  \le  {\displaystyle \frac{1}{4}}\left(N^{4}-2N^{3}-N^{2}+2N+8\right)e_{2,A},\;3\le N<\infty,
\end{equation*}
respectively.
\label{c2}
\end{corollary}

\begin{remark}
When $\mathbf{X}$ comprises independent
and identically distributed uniform random variables over $[0,1]$,
the results of the \emph{zero}-variate and univariate RDD approximations
presented in Theorems \ref{t4} and \ref{t5} coincide with those derived by Wang
\cite{wang08}. However, the results of the bivariate RDD approximation $-$
that is, Theorem \ref{t6} $-$ are new. Theorems \ref{t5} and \ref{t6} demonstrate that on average the error from the univariate RDD approximation eliminates
the variance terms associated with the univariate contribution. For
the bivariate RDD approximation, the variance portions resulting from
the univariate and bivariate terms have been removed as well. The
univariate and bivariate ADD approximations also satisfy this important
property. However, the coefficients of higher-variate terms in the
RDD errors are larger than unity, implying greater errors from RDD
approximations than from ADD approximations.
\end{remark}

\begin{remark}
From Corollary \ref{c1}, the \emph{zero}-variate
RDD approximation on average commits twice the amount of error as
does the \emph{zero}-variate ADD approximation. Since a \emph{zero}-variate
approximation, whether derived from ADD or RDD, does not capture the
random fluctuations of a stochastic response, the error analysis associated
with a \emph{zero}-variate approximation is useless. Nonetheless,
the \emph{zero}-variate results are reported here for completeness.
\end{remark}

\begin{remark}
Corollary \ref{c2} shows that the expected error
from the univariate RDD approximation is at least four times larger
than the error from the univariate ADD approximation. In contrast,
the expected error from the bivariate RDD approximation can be eight
times larger or more than the error from the bivariate ADD approximation.
Given a truncation, an ADD approximation is superior to an RDD approximation.
In addition, RDD approximations may perpetrate very large errors at
upper bounds when there exist a large number of variables and appropriate
conditions. For instance, consider a contrived example involving a
function of $100$ variables with a finite variance $\sigma^{2}>0$
and the following distribution of the variance terms: $\sum_{|u|=1}\sigma_{u}^{2}=0.999\sigma^{2}$,
$\sum_{2\le|u|\le99}\sigma_{u}^{2}=0$, and $\sum_{|u|=100}\sigma_{u}^{2}=0.001\sigma^{2}$.
Then, the errors from the univariate and bivariate ADD approximations
are both equal to $0.001\sigma^{2}$, which is negligibly small. In
contrast, the error from the univariate RDD approximation reaches
$(100^{2}-100+2)\times0.001\sigma^{2}\cong9.9\sigma^{2}$, an unacceptably
large magnitude already. Furthermore, the error from the bivariate
RDD approximation jumps to an enormously large value of $\frac{1}{4}(100^{4}-2\times100^{3}-100^{2}+2\times100+8)\times0.001\sigma^{2}\cong24,498\sigma^{2}$.
More importantly, the results reveal a theoretical possibility for
a higher-variate RDD approximation to commit a larger error than a
lower-variate RDD approximation $-$ an impossible scenario for the
ADD approximation. However, it is unlikely for this odd behavior to
be exhibited for realistic functions, where the variances of
higher-variate component functions attenuate rapidly or vanish altogether.
Nonetheless, a caution is warranted when employing RDD approximations
for stochastic analysis of high-dimensional systems.
\label{r1}
\end{remark}

\subsection{Expected Error from General Approximation}

The error analysis presented so far is limited to at most the bivariate approximation. In this subsection, the
approximation error from a general $S$-variate truncation is derived
as follows.

\begin{theorem}
Let $\mathbf{c}=(c_{1},\cdots,c_{N})\in\mathbb{R}^{N}$
be a random vector with the joint probability density function of
the form $f_{\mathbf{X}}(\mathbf{c})=\prod_{j=1}^{j=N}f_{j}(c_{j})$,
where $f_{j}$ is the marginal probability density function of its
$j$th coordinate. Then the expected error committed by the $S$-variate
RDD approximation for $0\le S<N,$ $S+1\le N<\infty$ is
\begin{equation}
\mathbb{E}\left[e_{S,R}(\mathbf{c})\right]={\displaystyle \sum_{s=S+1}^{N}\left[{\displaystyle 1+\sum_{k=0}^{S}}{\displaystyle \binom{s-S+k-1}{k}^{2}}\binom{s}{S-k}\right]\sum_{{\textstyle {\emptyset\neq u\subseteq\{1,\cdots,N\}\atop |u|=s}}}}\sigma_{u}^{2},\label{47}
\end{equation}
where $\sigma_{u}^{2}=\mathbb{E}\left[y_{u,A}^{2}(\mathbf{X}_{u})\right]$ is the variance of
the \emph{zero}-mean ADD component function $y_{u,A}$, $\emptyset\ne u\subseteq\{1,\cdots,N\}$.
\label{t7}
\end{theorem}

\begin{proof}
Expanding the square in (\ref{29}), the
expected error from the $S$-variate RDD approximation is
\begin{equation}
\mathbb{E}\left[e_{S,R}(\mathbf{c})\right]=I_{S,1}+I_{S,2}+I_{S,3},\label{48}
\end{equation}
where
\begin{equation}
\begin{split}
I_{S,1} & := \int_{\mathbb{R}^{2N}}y^{2}(\mathbf{x})f_{\mathbf{X}}(\mathbf{x})f_{\mathbf{X}}(\mathbf{c})d\mathbf{x}d\mathbf{c},\\
I_{S,2} & := -2\int_{\mathbb{R}^{2N}}y(\mathbf{x})\hat{y}_{S,R}(\mathbf{x};\mathbf{c})f_{\mathbf{X}}(\mathbf{x})f_{\mathbf{X}}(\mathbf{c})d\mathbf{x}d\mathbf{c},\\
I_{S,3} & := \int_{\mathbb{R}^{2N}}\hat{y}_{S,R}^{2}(\mathbf{x};\mathbf{c})f_{\mathbf{X}}(\mathbf{x})f_{\mathbf{X}}(\mathbf{c})d\mathbf{x}d\mathbf{c}
\end{split}\label{49}
\end{equation}
are three generic $2N$-dimensional integrals. The first integral
\begin{equation}
I_{S,1}=y_{\emptyset,A}^{2}+\sigma^{2}=y_{\emptyset,A}^{2}+\sum_{s=1}^{N}\:\sum_{{\textstyle {\emptyset\ne u\subseteq\{1,\cdots,N\}\atop |u|=s}}}\sigma_{u}^{2}\label{50}
\end{equation}
is the same as before. The second integral
\begin{equation}
\begin{split}
I_{S,2} & =  -2\int_{\mathbb{R}^{N}}y(\mathbf{x})\hat{y}_{S,A}(\mathbf{x})f_{\mathbf{X}}(\mathbf{x})d\mathbf{x}\\
 & =  -2\int_{\mathbb{R}^{N}}\hat{y}_{S,A}^{2}(\mathbf{x})f_{\mathbf{X}}(\mathbf{x})d\mathbf{x}\\
 & =  -2\mathbb{E}\left[\hat{y}_{S,A}^{2}(\mathbf{X})\right]=-2\left(y_{\emptyset,A}^{2}+\hat{\sigma}_{S,A}^{2}\right)\\
 & =  -2y_{\emptyset,A}^{2}-2{\displaystyle \sum_{s=1}^{S}\:\sum_{{\textstyle {\emptyset\ne u\subseteq\{1,\cdots,N\}\atop |u|=s}}}}\sigma_{u}^{2}
\end{split}\label{51}
\end{equation}
comprises variance terms associated with at most $S$ variables. The
third integral, using $\hat{y}_{S,R}(\mathbf{x};\mathbf{c})$
from (\ref{18}), takes the form
\begin{equation}
\begin{split}
I_{S,3} & = \int_{\mathbb{R}^{2N}}\biggl[{\displaystyle \sum_{k=0}^{S}}(-1)^{k}{\displaystyle \binom{N-S+k-1}{k}}{\displaystyle \sum_{{\textstyle {u\subseteq\{1,\cdots,N\}\atop |u|=S-k}}}}y(\mathbf{x}_{u},\mathbf{c}_{-u})\biggr]^{2}f_{\mathbf{X}}(\mathbf{x})f_{\mathbf{X}}(\mathbf{c})d\mathbf{x}d\mathbf{c},\\
 & =  b_{S}(0)y_{\emptyset,A}^{2}+{\displaystyle \sum_{s=1}^{N}b_{S}(s)\sum_{{\textstyle {\emptyset\ne u\subseteq\{1,\cdots,N\}\atop |u|=s}}}}\sigma_{u}^{2}
\end{split}\label{52}
\end{equation}
with the generic $s$-variate coefficient $b_{S}(s)$, $s=0,1,\cdots,N,$
for the $S$-variate approximation yet to be determined. Of $N+1$
such coefficients, the last one,
\begin{equation}
b_{S}(N)={\displaystyle \sum_{k=0}^{S}}\binom{N-S+k-1}{k}^{2}\binom{N}{S-k},\label{53}
\end{equation}
is easier to determine. It is obtained from the expansion coefficients of the square in (\ref{52}) that are associated with all variance terms
of $\sigma^{2}$. To determine other coefficients, $b_{S}(s)$, $s=0,1,\cdots,N-1$,
the procedure used before for the univariate or bivariate approximation is
unwieldy. An alternative scheme proposed here stems from the realization
that the expressions of those coefficients consist of terms from two
sources: (1) terms that depend solely on $N$, which can be described
by a function $f$, say, of $N$; and (2) terms that depend on both
$N$ and $s$, which can be described by another function $g$, say,
of $N$ and $s$. Following this rationale, let a generic coefficient be expressed by
\begin{equation}
b_{S}(s)=f(N)-g(N,s)\label{54}
\end{equation}
for any $0\le N<\infty$ and $0\le s\le N.$ Switching
the variables $N$ and $s$, (\ref{54}) produces
\begin{equation}
b_{S}(N)=f(s)-g(s,N).\label{55}
\end{equation}
At $s=N$, (\ref{54}) and (\ref{55}) result in $g(s,s)=g(N,N)=0$.
Either of these two equations at $s=N$ with $g(s,s)=g(N,N)=0$ in
mind yields $f(N)=b_{S}(N)$ or $f(s)=b_{S}(s)$, where the function
$b_{S}$ is already described in (\ref{53}). Therefore, the generic $s$-variate coefficient
\begin{equation}
b_{S}(s)={\displaystyle \sum_{k=0}^{S}}\binom{s-S+k-1}{k}^{2}\binom{s}{S-k},\;0\le s\le N,\label{56}
\end{equation}
where the binomial coefficients should be interpreted more generally
than their classical combinatorial definition, for instance,
\begin{equation}
 \binom{r}{k}:=
  \begin{cases}
{\displaystyle \frac{1}{k!}}r^{-\underline{k}}={\displaystyle \frac{1}{k!}}r(r-1)\cdots(r-k+1) & \text{if } k > 0, \\
   1       & \text{if } k = 0, \\
   0       & \text{if } k < 0,
  \end{cases}
\end{equation}
valid for any real number $r\in\mathbb{R}$ and any integer
$k\in\mathbb{Z}$. Adding all terms in (\ref{50}),
(\ref{51}), and (\ref{52}), with the cognizance that $b_{S}(s)=1$ for
all $0\le s\le S$, yields (\ref{47}), proving the theorem.
\end{proof}

\begin{corollary}
The lower and upper bounds of the
expected error $\mathbb{E}\left[e_{S,R}\right]$ from the $S$-variate
RDD approximation, expressed in terms of the error $e_{S,A}$ from
the $S$-variate ADD approximations, are
\begin{equation}
2^{S+1}e_{S,A}\le\mathbb{E}\left[e_{S,R}\right]\le\left[{\displaystyle 1+\sum_{k=0}^{S}}\binom{N-S+k-1}{k}^{2}\binom{N}{S-k}\right]e_{S,A},\label{57}
\end{equation}
$0\le S<N<\infty$, where the coefficients of the lower and upper bounds are obtained
from
\begin{equation}
1+b_{S}(S+1)={\displaystyle 1+\sum_{k=0}^{S}}\binom{S+1}{S-k}={\displaystyle \sum_{k=0}^{S+1}}\binom{S+1}{k}=2^{S+1}\label{58}
\end{equation}
and
\begin{equation}
1+b_{S}(N)={\displaystyle 1+\sum_{k=0}^{S}}\binom{N-S+k-1}{k}^{2}\binom{N}{S-k},\label{59}
\end{equation}
respectively.
\label{c3}
\end{corollary}

\begin{remark}
Both Theorem \ref{t7} and Corollary \ref{c3} are new and
provide a general result pertaining to RDD error analysis for an arbitrary
truncation. The specific results of the \emph{zero}-variate or univariate
or bivariate RDD approximation, derived in the preceding subsection,
can be recovered by setting $S=0$ or $1$ or $2$ in (\ref{47})
through (\ref{59}). From Corollary \ref{c3}, the expected error from the
$S$-variate RDD approximation of a multivariate function is at least
$2^{S+1}$ times larger than the error from the $S$-variate ADD approximation.
In other words, the ratio of RDD to ADD errors doubles for each increment
of the truncation. Consequently, ADD approximations are exceedingly
more precise than RDD approximations at higher-variate truncations.
\end{remark}

Although the relative disadvantage of using RDD over ADD worsens drastically with the truncation $S$, one hopes that the approximation error is also decreasing with
increasing $S$.  For instance, given a rate at which $\sigma_{u}^2$ decreases with $|u|$, what can be inferred on how fast $e_{S,A}$ and
$\mathbb{E}\left[e_{S,R}\right]$ decay with respect to $S$?  Corollary \ref{c4a} and subsequent discussions provide some insights.

\begin{corollary}
If the variance of the \emph{zero}-mean ADD component function $y_{u,A}$ diminishes according to $\sigma_{u}^2 \le Cp^{-|u|}$, where $\emptyset\ne u\subseteq\{1,\cdots,N\}$,
and $C>0$ and $p>1$ are two real-valued constants, then

\begin{equation}
e_{S,A} \le C{\displaystyle \sum_{s=S+1}^{N}{\displaystyle \binom{N}{s}}p^{-s}}\label{59c}
\end{equation}
and
\begin{equation}
\mathbb{E}\left[e_{S,R}\right] \le C{\displaystyle \sum_{s=S+1}^{N}\left[{\displaystyle 1+\sum_{k=0}^{S}}{\displaystyle \binom{s-S+k-1}{k}^{2}}\binom{s}{S-k}\right]
{\displaystyle \binom{N}{s}}p^{-s}}\label{59d}
\end{equation}
for $0\le S<N<\infty$.
\label{c4a}
\end{corollary}

When the equality holds in (\ref{59c}), $e_{S,A}$ decays strictly monotonically with respect to $S$ for any rate parameter $p$.  In contrast, $\mathbb{E}\left[e_{S,R}\right]$, according to (\ref{59d}), does not follow suit for an arbitrary $p$.  However, there exists a minimum threshold, say, $p_{\min}$, when crossed, $\mathbb{E}\left[e_{S,R}\right]$ also decays monotonically.  The threshold can be determined from the condition that  $\mathbb{E}\left[e_{0,R}\right] = \mathbb{E}\left[e_{1,R}\right]$, resulting in the relationship

\begin{equation}
{\displaystyle\frac{2}{p_{\min}}} = {\displaystyle \frac{(N-1)\left(1+\dfrac{1}{p_{\min}}\right)^N}{(1+p_{\min})^2} }\label{59e}
\end{equation}

\noindent between $N$ and $p_{\min}$.  Equation (\ref{59e}) supports an exact solution of

\begin{equation}
N = 1 + \frac{1}{ \ln \left(1+\dfrac{1}{p_{\min}}\right) } W\left( 2(1+p_{\min}) \ln \left(1+\dfrac{1}{p_{\min}}\right) \right)
\label{59f}
\end{equation}

\noindent in terms of $p_{\min}$, expressed employing the Lambert W function $W$ and can be inverted easily.  For instance, when $N=20$, (\ref{59f}) yields $p_{\min}=21.5187$, the only real-valued solution of interest. Depicted in Figure \ref{f1} (left), $p_{\min}$ derived from (\ref{59f}) increases monotonically and strikingly close to linearly with $N$ for the ranges of the variables examined.

Using the equalities in (\ref{59c}) and (\ref{59d}), Figure \ref{f1} (right) presents plots of two normalized errors, $\mathbb{E}\left[e_{S,R}\right]/\sigma^2$ and $e_{S,A}/\sigma^2$, against $S$, each obtained for $N=20$ and $p=5$ or $50$, where the variance $\sigma^2 = C [ ( 1+ \tfrac{1}{p} )^N - 1 ]$.  When the rate parameter is sufficiently low (\emph{e.g.}, $p=5< p_{\min}$), the expected RDD error initially rises before falling as $S$ becomes larger.  The non-monotonic behavior of the RDD error is undesirable, but it vanishes when the rate parameter is sufficiently high (\emph{e.g.}, $p=50>p_{\min}$).  No such anomaly is found in the ADD error for any $p$.

\begin{figure}[h]
\begin{centering}
\includegraphics[scale=0.62]{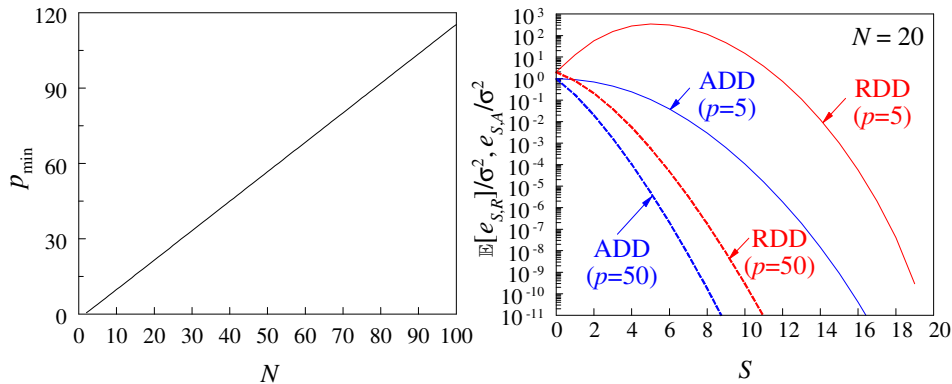}
\par\end{centering}
\caption{Relationship between $p_{\min}$ and $N$ (left) and normalized RDD and ADD errors versus $S$ for $N=20$, $p=5$ or $50$ (right).}
\label{f1}
\end{figure}

\begin{corollary}
The expected error $\mathbb{E}\left[e_{N-1,R}\right]$
from the best RDD approximation, expressed in terms of the error $e_{N-1,A}$
from the best ADD approximation, where the best approximations are
obtained by setting $S=N-1$, is
\begin{equation}
\mathbb{E}\left[e_{N-1,R}\right]=2^{N}e_{N-1,A},\;1\le N<\infty.\label{60}
\end{equation}
\label{c4}
\end{corollary}

\begin{remark}
Due to the factor $2^{N}$ in (\ref{60}), the expected error from the best RDD approximation as $N\to\infty$ can be significantly large unless the best ADD approximation commits an error equal to or smaller than $2^{-N}$.  In reference to Corollary \ref{c4a}, suppose that $\sigma_{u}^2 \le Cp^{-|u|}$.  Then $\mathbb{E}\left[e_{N-1,R}\right] \le C(2/p)^N$.  Therefore, $\mathbb{E}\left[e_{N-1,R}\right] \to 0$ as $N\to\infty$ for $p>2$.
\end{remark}

The error analysis presented in this paper pertains to only second-moment characteristics of $y(\mathbf{X})$.
Similar analyses or definitions aimed at higher-order moments or probability
distribution of $y$ can be envisioned, but no closed-form solutions
and simple expressions are possible. However, if $y$ satisfies the
requirements of the Chebyshev inequality or its descendants $-$ a
condition fulfilled by many realistic functions $-$ then
the results and findings from this work can be effectively exploited for stochastic analysis. See the longer version of the paper for further details.

\section{Conclusions}

Two variants of dimensional decomposition, namely, RDD and ADD, of
a multivariate function, both representing finite sums of lower-dimensional
component functions, were studied. The approximations resulting from
the truncated RDD and ADD are explicated, including clarifications
of parallel developments and synonyms used by various researchers.
For the RDD approximation, a direct form, previously developed by
the author's group, was found to provide a vital link to subsequent
error analysis. New theorems were proven about the expected errors
from the bivariate and general RDD approximations, so far available
only for the univariate RDD approximation,
when the reference point is selected randomly. They furnish new formulae
for the lower and upper bounds of the expected error committed by
an arbitrarily truncated RDD, providing a means to grade RDD against
ADD approximations. The formulae indicate that the expected error
from the $S$-variate RDD approximation of a function of $N$ variables,
where $0\le S<N<\infty$, is at least $2^{S+1}$ times larger than
the error from the $S$-variate ADD approximation. Consequently, ADD
approximations are exceedingly more precise than RDD approximations
at higher-variate truncations. The analysis also finds the RDD approximation
to be sub-optimal for an arbitrarily selected reference point, whereas
the ADD approximation always results in minimum error. Therefore,
RDD approximations should be used with caveat.

\bibliographystyle{amsplain}

\begin{thebibliography}{10}

\bibitem{hoeffding48}W. Hoeffding, \textit{A class of statistics with
asymptotically normal distributions}, Ann. Math. Statist. \textbf{19} (1948), 293-325.

\bibitem{owen03}A. B. Owen, \textit{The dimension distribution and
quadrature test functions}, Statist. Sinica \textbf{13} (2003),
1-17.

\bibitem{sobol69}I. M. Sobol, \textit{Multidimensional quadrature formulas
and Haar functions}, Nauka, Moscow, 1969 (In Russian).

\bibitem{sobol03}I. M. Sobol, \textit{Theorems and examples on high
dimensional model representations}, Reliab. Eng. Syst. Safe. \textbf{79} (2003), 187-193.

\bibitem{efron81}B. Efron and C. Stein, \textit{The Jackknife estimate
of variance}, Ann. Statist. \textbf{9} (1981), 586-596.

\bibitem{owen97}A. B. Owen, \textit{Monte Carlo variance of scrambled
net quadrature}, SIAM J. Numer. Anal. \textbf{34} (1997),
1884-1910.

\bibitem{hickernell95}F. J. Hickernell, \textit{Quadrature error bounds
with applications to lattice rules}, SIAM J. Numer. Anal. \textbf{33} (1996), 1995-2016. Corrected printing in \textbf{34} (1997),
853-866.

\bibitem{rabitz99}H. Rabitz and O. Alis. \textit{General foundations
of high-dimensional model representations}, J. Math. Chem., \textbf{25} (1999), 197-233.

\bibitem{rahman04}S. Rahman and H. Xu, \textit{A univariate dimension-reduction
method for multi-dimensional integration in stochastic mechanics}, Probabilist. Eng. Mech. \textbf{19} (2004), 393-408.

\bibitem{xu04}H. Xu and S. Rahman. \textit{A generalized dimension-reduction
method for multi-dimensional integration in stochastic mechanics},
Internat. J. Numer. Methods Engrg. \textbf{61}
(2004), 1992-2019.

\bibitem{xu05}H. Xu and S. Rahman, \textit{Decomposition methods for
structural reliability analysis}, Probabilist. Eng. Mech. \textbf{20} (2005), 239-250.

\bibitem{caflisch97}R. E. Caflisch, W. Morokoff, and A. Owen, \textit{Valuation
of mortgage backed securities using Brownian bridges to reduce effective
dimension}, J. Comput. Finance \textbf{1} (1997), 27-46.

\bibitem{wang08}X. Wang, \textit{On the approximation error in high
dimensional model representation}, In Proceedings of 40th Conference
on Winter Simulation, 453-462, 2008.

\bibitem{kuo10}F. Y. Kuo, I. H. Sloan, G. W. Wasilkowski, and H.
Wozniakowski, \textit{On decompositions of multivariate functions},
Math. Comp. \textbf{79} (2010), 953-966.

\bibitem{rahman08}S. Rahman, \textit{A polynomial dimensional decomposition
for stochastic computing}, Internat. J. Numer. Methods Engrg. \textbf{76} (2008), 2091-2116.

\bibitem{rahman09}S. Rahman, \textit{Extended polynomial dimensional
decomposition for arbitrary probability distributions}, J. Eng. Mech-ASCE \textbf{135} (2009), 1439-51.

\bibitem{hickernell04}F. J. Hickernell, I. H. Sloan, G. W. Wasilkowski,
\textit{On tractability of weighted integration over bounded and unbounded
regions in} $\mathbb{R}^{s}$, Math. Comp. \textbf{73}
(2004), 1885-1901.

\bibitem{griebel10}M. Griebel and M. Holtz, \textit{Dimension-wise
integration of high-dimensional functions with applications to finance},
J. Complexity \textbf{26} (2010), 455-489.

\bibitem{sobol01}I. M. Sobol, \textit{Global sensitivity indices for
nonlinear mathematical models and their Monte Carlo estimates}, Math. Comput. Simulation \textbf{55} (2001), 271-280.

\bibitem{rahman11}S. Rahman, \textit{Decomposition Methods for Structural Reliability Analysis Revisited},
Probabilist. Eng. Mech. \textbf{26} (2011), 357-363.

\end{thebibliography}

\end{document}